\documentclass[a4paper,12pt]{amsart}
\usepackage{amssymb,latexsym,graphicx,amsxtra,amsmath,amsthm}
\usepackage{hyperref}

\usepackage{graphics}
\usepackage[latin1]{inputenc}
\usepackage{multirow}

\usepackage{times}
\usepackage[T1]{fontenc}

\DeclareMathOperator{\RE}{Re}
\DeclareMathOperator{\IM}{Im}

\newcommand{\ti}{\widetilde}
\newcommand{\ze}{\zeta}
\newcommand{\la}{\lambda}
\newcommand{\ity}{\infty}
\newcommand{\C}{\mathbb{C}}
\newcommand{\R}{\mathbb{R}}
\newcommand{\N}{\mathbb{N}}
\newcommand{\Z}{\mathbb{Z}}

\numberwithin{equation}{section}
\newtheorem{theorem}{Theorem}[section]
\newtheorem{lemma}[theorem]{Lemma}
\newtheorem{corollary}[theorem]{Corollary}
\theoremstyle{remark}
\newtheorem{remark}[theorem]{Remark}
\newtheorem{example}[theorem]{Example}

\thanks {The research work of the  author is supported by research fellowship from Council of Scientific and Industrial Research (CSIR), New Delhi}

\begin{document}

\title[Escaping set and their dynamics]{On escaping sets of some families of entire functions and dynamics of composite entire functions}

\author[D. Kumar]{Dinesh Kumar}
\address{Department of Mathematics, University of Delhi,
Delhi--110 007, India}

\begin{abstract}
We consider two   families of functions $\mathcal F=\{f_{{\la},{\xi}}(z)= e^{-z+\la}+\xi: \la,\,\xi\in\C, \RE{\la}<0, \RE\xi\geq 1\}$ and $\mathcal F'=\{f_{{\mu},{\ze}}(z)= e^{z+\mu}+\ze: \mu,\,\ze\in\C, \RE{\mu}<0, \RE\ze\leq-1\}$ and investigate the escaping sets of members of the family $\mathcal F$ and $\mathcal F'.$ We also consider the dynamics of composite entire functions and provide conditions for equality of escaping sets of two transcendental entire functions.
\end{abstract}



\maketitle

\section{Introduction}\label{sec1}

 Let $f$ be a transcendental entire function. For $n\in\N,$ let $f^n$ denote the $n$-th iterate of $f.$ The set $F(f)=\{z\in\C : \{f^n\}_{n\in\N}\,\text{ is normal in some neighborhood of}\, z\}$ is called the Fatou set of $f$ or the set of normality of $f$ and its complement $J(f)$ is called  the Julia set of $f$. For an introduction to the properties of these sets  see \cite{bergweiler}. The escaping set of $f$ denoted by $I(f)$ is the set of points in the complex plane that tend to infinity under iteration of $f$. In general, it is neither an open nor a closed subset of $\C$ and has interesting topological properties. The escaping set for a transcendental entire function $f$ was studied for the first time by Eremenko \cite {e1} who established that
\begin{enumerate}
\item\ $I(f) \neq\emptyset;$
\item\ $J(f)=\partial I(f);$
\item\ $I(f)\cap J(f)\neq\emptyset;$
\item\ $\overline{I(f)}$ has no bounded components.
\end{enumerate}
In the same paper he stated the following conjectures:
\begin{enumerate}
\item[(i)] Every component of $I(f)$ is unbounded;
\item[(ii)] Every point of $I(f)$ can be connected to $\ity$ by a curve consisting of escaping points.
\end{enumerate}

For the exponential maps of the form $f(z)=e^z+\la$ with $\la>-1,$ it is known, by Rempe \cite{R3}, that the escaping set is a connected subset of the plane, and for $\la<-1,$ it is the disjoint union of uncountably many curves to infinity, each of which is connected component of $I(f)$ \cite{R4}. (These maps have no critical points and exactly one asymptotic value which is the omitted value $\la$). 

 In   \cite{SZ}, it was shown that every escaping point of every exponential map can be connected to $\ity$ by a curve consisting of escaping points.
Furthermore, it was also shown in \cite{SZ}  that if $f$ is an exponential map, that is, $f= e^{\la z},\,\la\in\C\setminus\{0\},$ then all components of $I(f)$ are unbounded, that is, Eremenko's conjecture \cite{e1} holds for exponential maps.

A complex number $w\in\C$ is a critical value of a transcendental entire function $f$ if there exist some $w_0\in\C$ with $f(w_0)=w$ and $f'(w_0)=0.$ Here $w_0$ is called a critical point of $f.$ The image of a critical point of $f$ is  critical value of $f.$ Also  $\zeta\in\C$ is an asymptotic value of a transcendental entire function $f$ if there exist a curve $\Gamma$ tending to infinity such that $f(z)\to \zeta$ as $z\to\ity$ along $\Gamma.$ 
Recall the Eremenko-Lyubich class
 \[\mathcal{B}=\{f:\C\to\C\,\,\text{transcendental entire}: \text{Sing}{(f^{-1})}\,\text{is bounded}\},\]
where Sing$f^{-1}$ is the set of critical values and asymptotic values of $f$ and their finite limit points. Each $f\in\mathcal{B}$ is said to be of bounded type. A transcendental entire function $f$ is of finite type if Sing$f^{-1}$ is a finite set. Furthermore, if the transcendental entire functions $f$ and $g$ are of bounded type then so is $f\circ g$ as Sing $((f\circ g)^{-1})\subset$ Sing $f^{-1}\cup f(\text{Sing}(g^{-1})),$ \cite{berg2}. Singularities of a transcendental map plays an important role in its dynamics. They are closely related to periodic components of the Fatou set \cite{morosowa}. For any transcendental entire function Sing$f^{-1}\neq\emptyset,$ \cite[p.\ 66]{Hua}. It is well known \cite{EL, keen}, if $f$ is of finite type then it has no wandering domains. Recently Bishop \cite{bishop} has constructed an example of a function of bounded type having a wandering domain. 
 In \cite{R1}, it was shown that if $f$ is an entire function of bounded type for which all singular orbits are bounded (that is, $f$ is postsingularly bounded), then each connected component of $I(f)$ is unbounded, providing a partial answer to a conjecture of Eremenko \cite{e1}.

Two functions $f$ and $g$ are called permutable if $f\circ g=g\circ f.$
Fatou  \cite{beardon}  proved that if $f$ and $g$ are two permutable rational functions  then $F(f)=F(g)$. This was an important result  that motivated the dynamics of composition of complex functions. Similar results for transcendental entire functions is still not known, though it holds in some very special cases   \cite[Lemma 4.5]{baker2}.
If $f$ and $g$ are transcendental entire functions, then so is $f\circ g$ and $g\circ f$ and the dynamics of one composite entire function helps in the study of the dynamics of the other and vice-versa. 
In \cite{dinesh2}, the authors considered the relationship between Fatou sets and singular values    of transcendental entire functions $f, g$ and $f\circ g$. They gave various conditions under which Fatou sets of $f$ and $f\circ g$ coincide and also considered relation between the singular values of $f, g$ and their compositions.
In \cite{dinesh3}, the authors have constructed several  examples where the dynamical behavior of $f$ and $g$ vary greatly from the dynamical behavior of $f\circ g$ and $g\circ f.$ Using approximation theory of entire functions, the authors have shown the existence of entire functions $f$ and $g$ having infinite number of domains satisfying various properties and relating it to their compositions. They explored and enlarged all the maximum possible ways of the solution in comparison to the past result worked out. 
 Recall that if $g$ and $h$ are transcendental entire functions and  $f$ is a continuous map of the complex plane into itself with $f\circ g=h\circ f,$ then $g$ and $h$ are said to be semiconjugated by $f$ and $f$ is called a semiconjugacy \cite{berg1}. In \cite{dinesh4}, the  author  considered the dynamics of semiconjugated entire functions and  provided several conditions under which the semiconjugacy  carries Fatou set of one entire function into  Fatou set of other entire function appearing in the semiconjugation. Furthermore, it was shown that under certain conditions on the growth  of  entire functions appearing in the semiconjugation, the set of  asymptotic values of the derivative of composition of the entire functions is bounded.

In this paper, we shall consider the two  families of functions $\mathcal F=\{f_{{\la},{\xi}}(z)= e^{-z+\la}+\xi: \la,\,\xi\in\C, \RE{\la}<0, \RE\xi\geq 1\}$ and $\mathcal F'=\{f_{{\mu},{\ze}}(z)= e^{z+\mu}+\ze: \mu,\,\ze\in\C, \RE{\mu}<0, \RE\ze\leq-1\}.$
 We  have given an explicit description of escaping sets of members of the families $\mathcal F$ and $\mathcal F'.$   For the family $\mathcal F,$ we have repeatedly used the fact that those points of the complex plane which land into the right half plane under iteration of functions in this family are not going to escape. Also, for the family  $\mathcal F',$ we have repeatedly used the fact that those points of the complex plane which land into the left half plane under iteration of functions in this family are not going to escape. 
  We have shown that for each $f\in\mathcal F,\,I(f)\subset\{z=x+iy : x<0, (4k-3)\frac{\pi}{2}<y<(4k-1)\frac{\pi}{2}, k\in\Z\}$ and for each $f\in\mathcal F',\,I(f)\subset\{z=x+iy : x>0, (4k-1)\frac{\pi}{2}<y<(4k+1)\frac{\pi}{2}, k\in\Z\}.$
We shall see that for each $f\in\mathcal F$ and for each $g\in\mathcal F',\,I(f)\cap I(g)=\emptyset.$  Moreover, we shall consider the dynamics of composite entire functions and provide conditions for equality of escaping sets of two transcendental entire functions. We have also investigated the relation between escaping sets of two conjugate entire functions.
\section{Theorems and their proofs}

\begin{theorem}\label{sec2,thm1}
 For each $f\in\mathcal F, I(f)$ is contained in $\{z=x+iy : x<0, (4k-3)\frac{\pi}{2}<y<(4k-1)\frac{\pi}{2}, k\in\Z\}.$
\end{theorem}

\begin{theorem}\label{sec2,thm2}
For each $f\in\mathcal F',\,I(f)\subset\{z=x+iy : x>0, (4k-1)\frac{\pi}{2}<y<(4k+1)\frac{\pi}{2}, k\in\Z\}.$
\end{theorem}

The proof of Theorem \ref{sec2,thm1} is  elementary. It is divided in several lemmas.   The main objective is to show that none of the points in the right half plane belongs to $I(f), f\in\mathcal F$. Using this notion, we try to find out the points in the left half plane,  which do not belong to $I(f),\,f\in\mathcal F.$ We observe that the right half plane is invariant under each $f\in\mathcal F.$ In similar spirit, the proof of Theorem \ref{sec2,thm2} is also elementary. It is also divided in several lemmas. Here too, the main objective is to show that none of the points in the left half plane belongs to $I(g), g\in\mathcal F'$. Using this notion, we try to find out the points in the right half plane,  which do not belong to $I(g),\,g\in\mathcal F'.$ We observe that the left half plane is invariant under each $g\in\mathcal F'.$

To prove Theorem \ref{sec2,thm1}, we first prove the following lemmas:

\begin{lemma}\label{sec2,lem1}
The set $I_1=\{z\in\C : \RE z >0, \RE(e^{-z+\la})>0,\,\la\in\C,\,\text{with}\,\RE\la<0\}$ does not intersect $I(f)$ for each $f\in\mathcal F.$
\end{lemma}

\begin{proof}
Observe that $\RE(e^{-z+\la})> 0,$ implies $\dfrac{(4k-1)\pi}{2}+\IM\la<y<\dfrac{(4k+1)\pi}{2}+\IM\la,$ where $k\in\Z.$ Therefore, the set $I_1$ is the entire right half plane $\{z:\RE z>0\}.$ We show that no point in $I_1$ escapes to $\ity$ under iteration of each $f\in\mathcal F.$ For  this we show  $|f^k(z)|\leq 1+|\xi|$ for all $k\in\N,\, z\in I_1.$  Suppose on the contrary  there exist $n\in\N$ and $z\in I_1$ such that $|f^n(z)|>1+|\xi|.$ Then $|f(f^{n-1}(z))|> 1+|\xi|,$  implies $1+|\xi|<|e^{-f^{n-1}(z)+\la}+\xi|.$ This shows that $e^{-\RE f^{n-1}(z)+\RE\la}>1$, and as $\RE\la<0,$ we obtain  $-\RE f^{n-1}(z)>0,$ that is, $\RE f^{n-1}(z)<0.$  Further this implies that $\RE(f(f^{n-2}(z)))<0,$ that is, $\RE(e^{-f^{n-2}(z)+\la}+\xi)<0,$ which implies  $-\RE(e^{-f^{n-2}(z)+\la})> \RE\xi\geq 1.$  Since $|z|\geq -\RE z $ for all $z\in\C,$ we get $|(e^{-f^{n-2}(z)+\la}|>1,$ that is,  $e^{-\RE f^{n-2}(z)+\RE\la}>1$ and so  $\RE f^{n-2}(z)<0.$ By induction  we will get $\RE f(z)<0.$ But $\RE f(z)=\RE(e^{-z+\la})+\RE\xi>{0+1}=1,$  so we arrive at  a contradiction and therefore proves the assertion.
\end{proof}

\begin{lemma}\label{sec2,lem2}
The set $I_2=\{z\in\C: \RE z=0\}$ does not intersect $I(f)$ for each $f\in\mathcal F.$
\end{lemma}

\begin{proof}
Any $z\in I_2$ has the form $z=iy$ for some $y\in\R.$ Now $\RE f(z)=\RE (e^{-iy+\la}+\xi)=e^{\RE\la}\cos(y-\IM\la)+\RE\xi,$ and as $\RE\xi\geq 1$ we get $\RE f(z)>0.$ From above lemma, $I_2\cap I(f)=\emptyset,$ for each $f\in\mathcal F$ and hence the result.
\end{proof}

\begin{lemma}\label{sec2,lem3}
The set $I_3=\{z\in\C: \RE z<0, \RE(e^{-z+\la})>0\}$ does not intersect $I(f)$ for each $f\in\mathcal F.$ 
\end{lemma}

\begin{proof}
For each $f\in\mathcal F$ and for each $z\in I_3,$ $f(z)$ belongs to the right half plane and hence cannot escape to $\ity$ using Lemma \ref{sec2,lem1}.
\end{proof}
The proof of Theorem \ref{sec2,thm1} now follows from the above three lemmas.

\begin{remark}\label{sec2,rem1'}
The right half plane is invariant under each $f\in\mathcal F$.
\end{remark}

To prove Theorem \ref{sec2,thm2}, we first prove the following lemmas:

\begin{lemma}\label{sec2,lem4}
The set $I_1'=\{z\in\C: \RE z <0, \RE(e^{z+\mu})<0,\,\mu\in\C\,\text{with}\,\RE\mu<0\}$ does not intersect $I(f)$ for each $f\in\mathcal F'.$
\end{lemma}

\begin{proof}
 Observe that $\RE(e^{z+\mu})< 0,$ implies $\dfrac{(4k-3)\pi}{2}-\IM\mu<y<\dfrac{(4k-1)\pi}{2}-\IM\mu,$ where $k\in\Z.$ Therefore,  the set $I_1'$ is the entire left half plane $\{z:\RE z<0\}.$ We show no point in $I_1'$ escapes to $\ity$ under iteration of each $f\in\mathcal F'.$ For  this we show  $|f^k(z)|\leq 1+|\ze|$ for all $k\in\N,\,z\in I_1'.$  Suppose on the contrary  there exist $n\in\N$ and $z\in I_1'$ such that $|f^n(z)|>1+|\ze|.$
Now $|f(f^{n-1}(z))|> 1+|\ze|,$  implies $1+|\ze|<|e^{f^{n-1}(z)+\mu}+\ze|.$ This shows that $e^{\RE f^{n-1}(z)+\RE\mu}>1$, and as $\RE\mu<0,$ we obtain  $\RE f^{n-1}(z)>0.$   Further this implies that $\RE(f(f^{n-2}(z)))>0,$ that is, $\RE(e^{f^{n-2}(z)+\mu}+\ze)>0,$ which implies $\RE(e^{f^{n-2}(z)+\mu})>-\RE\ze\geq 1.$  Since $|z|\geq \RE z$ for all $z\in\C,$ we get $|(e^{f^{n-2}(z)+\mu})|>1,$ that is,  $e^{\RE f^{n-2}(z)+\RE\mu}>1$ and so $\RE f^{n-2}(z)>0.$ By induction  we will get $\RE f(z)>0.$ But $\RE f(z)=\RE(e^{z+\mu})+\RE\ze<{0-1}=-1,$  so we arrive at  a contradiction and therefore proves the assertion.
\end{proof}

\begin{lemma}\label{sec2,lem5}
The set $I_2'=\{z\in\C: \RE z=0\}$ does not intersect $I(f)$ for each $f\in\mathcal F'.$
\end{lemma}

\begin{proof}
Any $z\in I_2'$ has the form $z=iy$ for some $y\in\R.$ Now $\RE f(z)=\RE (e^{iy+\mu}+\ze)=e^{\RE\mu}\cos(y+\IM\mu)+\RE\ze,$ and as $\RE\ze\leq -1$ we get $\RE f(z)<0.$ From above lemma, $I_2'\cap I(f)=\emptyset,$ for each $f\in\mathcal F'$ and hence the result.
\end{proof}

\begin{lemma}\label{sec2,lem6}
The set $I_3'=\{z\in\C: \RE z>0, \RE(e^{z+\mu})<0\}$ does not intersect $I(f)$ for each $f\in\mathcal F'.$ 
\end{lemma}

\begin{proof}
For each $f\in\mathcal F'$ and for each $z\in I_3',$ $f(z)$ belongs to the left half plane and hence cannot escape to $\ity$ using Lemma \ref{sec2,lem4}.
\end{proof}
The proof of Theorem \ref{sec2,thm2} now follows from the above three lemmas.

\begin{remark}\label{sec2,rem2'}
The left half plane is invariant under each $f\in\mathcal F'.$
\end{remark}

The following corollary is immediate

\begin{corollary}\label{sec2,cor1}
For each $f\in\mathcal F$ and for each $g\in\mathcal F', I(f)\cap I(g)=\emptyset.$
\end{corollary}

However, if the entire functions $f$ and $g$ share some relation, then their escaping sets do intersect. For instance, we have

\begin{theorem}\label{sec2,thmx'}
Let $f$ be a transcendental entire function  of period $c,$  and let $g=f^s+c,$\,  $s\in\N.$ Then $I(f)=I(g).$
\end{theorem}

\begin{proof}
For $n\in\N,\,g^n=f^{ns}+c$ and so  $I(g)=I(f).$
\end{proof}

We illustrate this with an example.

\begin{example}\label{sec2,egi}
Let $f=e^{\la z},\,\la\in\C\setminus\{0\}$ and $g=f^s+\frac{2\pi i}{\la},\,s\in\N.$  For $n\in\N,\,g^n=f^{ns}+\frac{2\pi i}{\la}$ and so  $I(g)=I(f).$  
\end{example}

\section{Composite entire functions and their dynamics}\label{sec3}

In this section, we  prove some results related to escaping sets of composite entire functions.
Recall that if a transcendental entire function $f$ is of bounded type, then $I(f)\subset J(f)$ and $J(f)=\overline{I(f)}$ \cite{EL}.

\begin{theorem}\label{sec2,thm3}
If $f$ and $g$ are permutable transcendental entire functions of bounded type, then $\overline{I(f)}$ and $\overline{I(g)}$ are completely invariant under $f\circ g.$
\end{theorem}
\begin{proof}
From \cite{baker2}, we have $g(J(f))\subset J(f)$ and so $g(\overline{I(f)})\subset\overline{I(f)}.$ From \cite{berg1},\, $g^{-1}(\overline{I(f)})\subset\overline{I(f)}.$ Hence $\overline{I(f)}$ is completely invariant under $g.$ On similar lines, $\overline{I(g)}$ is completely invariant under $f.$ As $J(f)=\overline{I(f)}$ is completely invariant under $f$ and $J(g)=\overline{I(g)}$ is completely invariant under $g,$ we have $\overline{I(f)}$ and $\overline{I(g)}$ are both completely invariant under $f$ and $g$ respectively and this completes the proof of  the theorem.
\end{proof}
We next prove an important lemma which will be used heavily in the results to follow.
\begin{lemma}\label{sec2,lemx'}
Let $f$ and $g$ be   transcendental  entire functions satisfying $f\circ g=g\circ f.$ Then $F(f\circ g)\subset F(f)\cap F(g).$
\end{lemma}
\begin{proof}
In \cite{berg2}, it was shown that $z\in F(f\circ g)$ if and only if $f(z)\in F(g\circ f).$ Since $f\circ g=g\circ f,\, F(f\circ g)$ is completely invariant under $f$ and by symmetry, under $g$ respectively and so, in particular, it is forward invariant under them. So $f(F(f\circ g))\subset F(f\circ g)$ and $g(F(f\circ g))\subset F(f\circ g),$  which by Montel's Normality Criterion implies $F(f\circ g)\subset F(f)$ and $F(f\circ g)\subset F(g)$ and hence the result. 
\end{proof}

\begin{theorem}\label{sec2,thmx'}
Let $f$ and $g$ be   transcendental  entire functions of bounded type satisfying $f\circ g=g\circ f.$ Then $\overline{I(f)}\cup\overline{I(g)}\subset \overline{I(f\circ g)}.$
\end{theorem}
\begin{proof}
Let $z_0\notin\overline{I(f\circ g)}.$ Then there exist a neighborhood $U$ of $z_0$ such that $U\cap I(f\circ g)=\emptyset.$ As $f\circ g$ is of bounded type, we get $U\subset F(f\circ g).$ From Lemma \ref{sec2,lemx'}, $U\subset F(f)$ and $U\subset F(g)$. Therefore, $U\cap I(f)=\emptyset$ and $U\cap I(g)=\emptyset$. Thus $z_0\notin \overline{I(f)}\cup\overline{I(g)}$ and this proves the result.
\end{proof}


\begin{theorem}\label{sec2,thmx''}
Let $f$ and $g$ be   transcendental  entire functions satisfying $f\circ g=g\circ f.$ Then
\begin{enumerate}
\item [(i)] $I(f\circ g)$ is completely invariant under $f$ and $g$ respectively;
\item [(ii)] $I(f\circ g)\subset I(f)\cup I(g);$
\item [(iii)] For any two positive integers $i$ and $j,$ $I(f^i\circ g^j)=I(f\circ g).$
\end{enumerate}
\end{theorem}
\begin{proof}
\begin{enumerate}
\item [(i)] We first show that $z\in I(f\circ g)$ if and only if $g(z)\in I(g\circ f).$ Let $z\in I(f\circ g).$ Then $(f\circ g)^n(z)\to\ity$ as $n\to\ity,$ that is,
$f((g\circ f)^{n-1}g(z))\to\ity$ as $n\to\ity.$ As $f$ is an entire function, this implies that $(g\circ f)^{n-1}g(z)\to\ity$ as $n\to\ity,$ that is, $g(z)\in I(g\circ f).$ On the other hand, let $g(z)\in I(g\circ f).$ Then $(g\circ f)^{n}(g(z))\to\ity$ as $n\to\ity,$ that is, $g((f\circ g)^n(z))\to\ity$ as $n\to\ity.$ Again, as $g$ is entire, this forces $(f\circ g)^n(z)\to\ity$ as $n\to\ity.$  So, $z\in I(f\circ g)$ which proves the claim. As $f\circ g=g\circ f,$ we obtain $z\in I(f\circ g)$ if and only if $g(z)\in I(f\circ g)$ which implies $I(f\circ g)$ is completely invariant under $g,$ and by symmetry, under $f$ respectively.

\item [(ii)] Suppose $z_0\notin I(f)\cup I(g)$. 
Then both $f^n(z_0)$ and $g^n(z_0)$ are bounded as $n\to\ity,$ which in turn  implies $(f\circ g)^n(z_0)$ is bounded as $n\to\ity$  and hence the result.

\item [(iii)] For $i, j\in\N,$ assume $i\geq j.$ We first show that $I(f^i\circ g^j)\subset I(f\circ g).$ To this end, let $w\notin I(f\circ g).$ Then $(f\circ g)^n(w)$ is bounded as $n\to\ity,$ which in turn (using a diagonal sequence argument) implies that $(f^i\circ g^j)^n(w)$ is bounded as  $n\to\ity.$ On similar lines, we get $I(f\circ g)\subset I(f^i\circ g^j)$ and hence $I(f^i\circ g^j)=I(f\circ g)$ for all $i, j\in\N.$\qedhere  
\end{enumerate}
\end{proof}

\begin{theorem}\label{sec2,thmx'''}
Let $f$ and $g$ be   transcendental  entire functions satisfying $f\circ g=g\circ f.$ Then $g(I(f))\supset I(f).$
\end{theorem}
\begin{proof}
Let $w\notin I(f).$ Then $f^n(w)$ is bounded and so $g(f^n(w))$ is bounded, which implies $g(w)\notin I(f)$ which proves the result.
\end{proof}

We now provide an important criterion for the equality of escaping sets for two entire functions.
\begin{theorem}\label{sec2,thmx''''}
Let $f$ and $g$ be two  transcendental  entire functions of bounded type satisfying $f\circ g=g\circ f.$ Assume for each $w\in \overline{I(g)}$ and for a sequence $\{w_n\}\subset I(g)$ converging to $w,\;\lim_{n\to\ity}\lim_{k\to\ity}g^k(w_n)=\lim_{k\to\ity}\lim_{n\to\ity}g^k(w_n).$ Then $I(f)=I(g).$
\end{theorem}
\begin{proof}
From \cite[Lemma 5.8]{dinesh1}, $F(f)=F(g)$ and so  $J(f)=J(g)$ which implies $\overline{I(f)}=\overline{I(g)}.$ Let $w\in I(f).$ Then there exist a sequence $\{w_n\}\subset I(g)$ such that $w_n\to w$ as $n\to\ity.$ For each $n\in\N, g^k(w_n)\to\ity$ as $k\to\ity.$ Now taking limit as $n$ tends to $\ity,$ and interchanging the two limits (by hypothesis) we obtain, $g^k(w)\to\ity$ as $k\to\ity$ which implies $w\in I(g)$ and so $I(f)\subset I(g).$ On similar lines, one obtains $I(g)\subset I(f)$ and this completes the proof of the result. 
\end{proof}

\begin{remark}\label{sec2,remx''}
The  result, in particular, establishes one of Eremenko's conjecture \cite{e1} that every component of $I(f)$ is unbounded.
\end{remark}

We now provide some conditions under which $\overline{I(f)}$ equals $\overline{I(f\circ g)}.$
\begin{theorem}\label{sec2,thmx'''''}
Let $f$ and $g$ be two  transcendental  entire functions. Then the following holds: 
\begin{enumerate}
\item [(i)]  If $f$ and $g$ are permutable and  of bounded type then $\overline{I(f)}=\overline{I(f\circ g)};$
\item [(ii)] If $f$ is of period $c$ and $g=f^m+c\,$ for some $m\in\N,$ then $\overline{I(f)}=\overline{I(f\circ g)}.$
\end{enumerate}
\end{theorem}

\begin{proof}
\begin{enumerate}
\item [(i)] In view of Theorem \ref{sec2,thmx''}\,(ii), it suffices to show that $\overline{I(f)}\subset \overline{I(f\circ g)}.$ To this end, let $w\notin \overline{I(f\circ g)}.$ Then there exist a neighborhood $U$ of $w$ such that $U\cap I(f\circ g)=\emptyset.$ As $f\circ g$ is of bounded type, it follows that $U\subset F(f\circ g)$ and so from Lemma \ref{sec2,lemx'}, $U\subset F(f).$ Therefore, $U\cap I(f)=\emptyset$ which implies $w\notin \overline{I(f)}$ and this proves the result.
\item [(ii)] Observe that $f\circ g(z)=f^{m+1}(z)$ and hence the result.\qedhere
\end{enumerate}
\end{proof}

\begin{remark}\label{sec3,remy'}
Combining Theorem \ref{sec2,thmx'} and Theorem \ref{sec2,thmx'''''}(i), and using Theorem \ref{sec2,thmx''}(ii) we get that if $f$ and $g$ are permutable and  of bounded type, then  $\overline{I(f\circ g)}=\overline{I(f)}\cup\overline{I(g)}.$
\end{remark}

Finally, we discuss the relation between the escaping sets of two conjugate entire functions. Recall that two entire functions $f$ and $g$ are conjugate if there exist a conformal map $\phi:\C\to\C$ with $\phi\circ f=g\circ\phi.$ By a conformal map $\phi:\C\to\C$ we mean an analytic and univalent map of the complex plane $\C$ that is exactly of the form $az+b,$ for some non zero $a.$
If $f$ and $g$ are two rational functions which are conjugate under some Mobius transformation $\phi:\ti\C\to\ti\C$, then it is well known \cite[p.\ 50]{beardon}, $\phi(J(f))=J(g).$ This gets easily carried over to transcendental entire funtions which are conjugate under a conformal map $\phi:\C\to\C$. Moreover, if $f$ is of bounded type which is conjugate under the conformal map $\phi$ to an entire function $g,$ then $g$ is also of bounded type and  $\phi(\overline{I(f)})=\overline{I(g)}.$ More generally, if transcendental entire functions $f$ and $g$ are conjugate by conformal map $\phi,$ then $\phi(I(f))=I(g).$


\begin{thebibliography}{00}


\bibitem{baker2} I. N. Baker, Wandering domains in the iteration of entire functions, Proc. London Math. Soc. \textbf{49} (1984), 563-576.
\bibitem{beardon} A. F. Beardon, \emph{Iteration of rational functions}, Springer Verlag, (1991).
\bibitem{bergweiler} W. Bergweiler, Iteration of meromorphic functions, Bull. Amer. Math. Soc. \textbf{29} (1993), 151-188.
\bibitem{berg1} W. Bergweiler and A. Hinkannen, On the semiconjugation of entire functions, Math. Seminar Christian Albrechts-Univ. Kiel \textbf{21}(1997), 1-10.
\bibitem{berg2} W. Bergweiler and Y. Wang, On the dynamics of composite entire functions. Ark. Math. \textbf{36} (1998), 31-39.
\bibitem{bishop} C. J. Bishop, Constructing entire functions by quasiconformal folding, preprint (2011).

\bibitem{e1} A. E. Eremenko, On the iteration of entire functions, Ergodic Theory and Dynamical Systems, Banach Center Publications \textbf{23}, Polish Scientific Publishers, Warsaw, (1989), 339-345.
\bibitem{EL} A. E. Eremenko and M. Yu. Lyubich, Dynamical properties of some classes of entire functions, Ann. Inst. Fourier, Grenoble, \textbf{42} (1992), 989-1020.
\bibitem{keen} L. R. Goldberg and L. Keen, A finiteness theorem for a dynamical class of entire functions, Ergodic Theory and Dynamical Systems, \textbf{6} (1986), 183-192.
\bibitem{Hua} X. H. Hua, C. C. Yang, \emph{Dynamics of transcendental functions}, Gordon and Breach Science Pub. (1998).
\bibitem{dinesh1} D. Kumar and S. Kumar, The dynamics of semigroups of transcendental entire functions I, arXiv:math.DS/13027249, (2013) (accepted for publication in Indian J. Pure Appl. Math.)
\bibitem{dinesh2}  D. Kumar and S. Kumar, On dynamics of composite entire functions and singularities, Bull. Cal. Math. Soc. \textbf{106} (2014), 65-72.

\bibitem{dinesh3} D. Kumar, G. Datt and S. Kumar, Dynamics of composite entire functions, arXiv:math.DS/12075930, (2013) (accepted for publication in J. Ind. Math. Soc.)
\bibitem{dinesh4} D. Kumar, On dynamics of semiconjugated entire functions,  J. Ind. Math. Soc. \textbf{82} (2015), 53-59.
\bibitem{morosowa} S. Morosawa, Y. Nishimura, M. Taniguchi and T. Ueda, \emph{Holomorphic dynamics}, Cambridge Univ. Press, (2000).
\bibitem{R1} L. Rempe, On a question of Eremenko concerning escaping sets of entire functions, Bull. London Math. Soc. \textbf{39:4}, (2007), 661-666.
\bibitem{R3} L. Rempe, The escaping set of the exponential, Ergodic Theory and Dynamical Systems, \textbf{30} (2010), 595-599.
\bibitem{R4} L. Rempe, Connected escaping sets of exponential maps,  Ann. Acad. Sci. Fenn. Math. \textbf{36} (2011), 71-80.
\bibitem{RRRS} G. Rottenfusser, J. Ruckert, L. Rempe and D. Schleicher, Annals of Mathematics, \textbf{173}, (2011), 77-125

\bibitem{SZ} D. Schleicher and J. Zimmer, Escaping points of exponential maps, J. London Math. Soc. (2) \textbf{67} (2003), 380-400.






\end{thebibliography}
\end{document}